\documentclass[a4paper,12pt]{amsart} 

\usepackage{amssymb,amsmath}
\usepackage{pb-diagram}
\usepackage{pb-xy}
\usepackage[cmtip,arrow]{xy}

\newtheorem{thm}{Theorem}[section] 
\newtheorem{prop}[thm]{Proposition}
\newtheorem{cor}[thm]{Corollary} 
\newtheorem{lem}[thm]{Lemma}

\theoremstyle{definition} 
 
\newtheorem{rem}[thm]{Remark}

\newtheorem*{ackn}{Acknowledgment}

\numberwithin{equation}{section}

\newcommand{\skal}[2]{\langle #1,#2\rangle}

\begin{document}

\title{On the frame bundle adapted to a submanifold}
\author{Kamil Niedzia\l omski}


\address{
Department of Mathematics and Computer Science \endgraf
University of \L\'{o}d\'{z} \endgraf
ul. Banacha 22, 90-238 \L\'{o}d\'{z} \endgraf
Poland
}
\email{kamiln@math.uni.lodz.pl}

\begin{abstract}
Let $M$ be a submanifold of a Riemannian manifold $(N,g)$. $M$ induces a subbundle $O(M,N)$ of adapted frames over $M$ of the bundle of orthonormal frames $O(N)$. Riemannian metric $g$ induces natural metric on $O(N)$. We study the geometry of a submanifold $O(M,N)$ in $O(N)$. We characterize the horizontal distribution of $O(M,N)$ and state its correspondence with the horizontal lift in $O(N)$ induced by the Levi--Civita connection on $N$.  

In the case of extrinsic geometry, we show that minimality is equivalent to harmonicity of the Gauss map of the submanifold $M$ with deformed Riemannian metric. In the case of intrinsic geometry we compute the curvatures and compare this geometry with the geometry of $M$.
\end{abstract}

\maketitle

\section{Introduction}

It is a basic fact that the reduction of the connection on the frame bundle $L(N)$ to the frame bundle $O(N)$ of orthonormal frames over a Riemannian manifold on $N$ is equivalent with the condition for the connection $\nabla$ on $N$ to be Riemannian. 

If we consider a submanifold $M$ of $N$ we may introduce the bundle $O(M,N)$ of adapted frames, which is a subbundle of $O(N)|M$. It appears that the Levi--Civita connection on $N$ is not reducible to $O(M,N)$. In fact the connection
\begin{equation*}
\nabla '_XY=(\nabla_XY^{\top})^{\top}+(\nabla_X Y^{\bot})^{\bot}
\end{equation*}
on $TN|M$ is reducible to $O(M,N)$. We define the horizontal distribution on $O(M,N)$ induced by $\nabla '$. We show the relation between horizontal lifts $X^{h'}$ and $X^h$ to $O(M,N)$ and to $O(N)$, induced by $\nabla '$ and $\nabla$, respectively. The correspondence relies on the difference
\begin{equation*}
S_XY=\nabla_X Y-\nabla '_XY=(\nabla_XY^{\top})^{\bot}+(\nabla_X Y^{\bot})^{\top}=\Pi(X,Y^{\top})-A_{Y^{\bot}}(X),
\end{equation*}
where $\Pi$ denotes the second fundamental form of $M$ and $A_V$ the Weingarten operator in the normal direction $V$. It is not surprising, since the complement $\mathfrak{m}$ of the Lie algebra $\mathfrak{h}=o(n)\oplus o(p)$ of structure group of $O(M,N)$ in the Lie algebra $\mathfrak{g}=o(n+p)$ of the structure group of $O(N)$ consists of block matrices of the form
\begin{equation*}
\left( \begin{array}{cc} 0 & A \\ -A^{\top} & 0 \end{array} \right),
\end{equation*}  
where we put $n=\dim M$, $n+p=\dim N$.

We study the intrinsic and extrinsic geometry of $O(M,N)$ in $O(N)$. In order to do it we equip $O(N)$, and more generally $L(N)$, with the Sasaki--Mok metric induced by the Riemannian metric on $N$. This metric agrees with the canonical Riemannian metric on $O(N)$ defined on the vertical subspace with the use of the Killing form on the Lie algebra $o(n+p)$.

We begin by introducing in the frame bundle $L(N)$ special vertical vector fields 
$\bar{T}$ induced from the tensor fields $T:TN\to TN$. These vector fields generate the space of all invariant vertical vector fields and they are tangent to $O(N)$ if and only if tensors $T$ are skew--symmetric. Then the Levi--Civita connection of $O(N)$ on these vector fields is just the bi--invariant connection, namely
\begin{equation*}
\nabla^{O(M)}_{\bar{T}} \bar{T'}=\frac{1}{2}[\bar{T},\bar{T'}].
\end{equation*} 

For the extrinsic geometry, notice that the quotient $O(n+p)/(O(n)\oplus O(p))$ of structure groups of $O(N)$ and $O(M,N)$, respectively, is the fiber of the Grassmann bundle $Gr_n(N)$ of $n$--dimensional subspaces of $TN$. The tangent space of this quotient is the complement of $o(n)\oplus o(p)$ in $o(n+p)$ and, simultaneously, the transversal of $TO(M,N)$ in $TO(N)$. This suggests the correspondence between the extrinsic geometry of a submanifold $O(M,N)$ in $O(N)$ and the geometry of $Gr_n(N)$.  

Indeed, the main theorem gives the characterization of minimality of $O(M,N)$, which is equivalent to the harmonicity of the Gauss map of a submanifold with respect modified Riemannian metric
\begin{equation*}
\tilde{g}(X,Y)=g(PX,Y),
\end{equation*}
where
\begin{equation*}
P(X)=X-2\sum_A S_{e_A}^2(X),\quad X\in TM.
\end{equation*}

\begin{thm}
Submanifold $O(M,N)$ is minimal in $O(N)$ if and only if the Gauss map $\gamma:(M,\tilde{g})\to (Gr_p(N),g_{Gr})$ of the submanifold $M\subset N$ is harmonic, where $g_{Gr}$ is the Riemannian metric induced from the Riemannian metric $g$ on $M$.
\end{thm} 

For the intrinsic geometry, we compute curvature tensor and sectional curvatures. We state the correspondence between the geometric properties of $M$ and $O(M,N)$. It is interesting that the Levi--Civita connection on the horizontal lifts depends on the Levi--Civita connection $\tilde{\nabla}$ of $\tilde{g}$, namely
\begin{equation*}
\nabla^{O(M,N)}_{X^{h'}}Y^{h'}=\left(\tilde{\nabla}_XY\right)^{h'}-\overline{R '(X,Y)},
\end{equation*}
where $R'$ is the curvature tensor of $\nabla '$.

Similar results to the ones obtained in this paper can be generalized to the frame bundle adapted to a foliation or merely to a distribution. In that case the adapted frame bundle is a bundle over $N$ and the connection $\nabla '$ is a connection on $N$. These results are studied by the author independently \cite{nie2}.

Throughout the paper we use the following index convention:
\begin{equation*}
1\leq A,B,C\leq p,\quad p+1\leq \alpha,\beta,\gamma\leq p+n,\quad 1\leq i,j,k\leq p+n.
\end{equation*}

\begin{ackn}
The author wishes to thank Wojciech Koz\l owski and Pawe{\l} Walczak for fruitful conversations. The author was supported by the Polish NSC grant No. 6065/B/H03/2011/40
\end{ackn}

\section{Frame bundle, Sasaki--Mok metric, induced vector fields}

Let $L(N)$ be the bundle of frames over a Riemannian manifold $N$. Let $\omega$ be the connection form of the Levi--Civita connection $\nabla$. The horizontal distribution $\mathcal{H}$ is the kernel of $\omega$, whereas the vertical distribution $\mathcal{V}$ is the kernel of the differential of projection $\pi:L(N)\to N$. Thus 
\begin{equation*}
\mathcal{H}_u=\ker\omega_u,\quad \mathcal{V}_u=\ker\pi_{\ast u}\quad\textrm{and}\quad 
T_uL(N)=\mathcal{H}_u\oplus\mathcal{V}_u.
\end{equation*}   

Vertical distribution consists of fundamental vector fields $A^{\ast}_u$, where $A\in gl(n)$. Denote by $\mathcal{V}^i_u$ the subspace of $\mathcal{V}_u$ spanned by elements $(E^j_i)^{\ast}_u$, $j=1,\ldots,p+n$. Then
\begin{equation*}
T_uL(N)=\mathcal{H}_u\oplus\mathcal{V}^1_u\oplus\ldots\oplus\mathcal{V}^{p+n}_u.
\end{equation*}
Each summand is $(p+n)$--dimensional and any vector $X\in TN$ has the unique lift $X^{v,i}_u$ to the subspace $\mathcal{V}^i_u$. More precisely, for a vector $X$ and a basis $u\in L(N)$ we put
\begin{equation*}
X^{v,i}_u=\sum_j \xi^j (E^j_i)^{\ast}_u,
\end{equation*}
where $\xi_i$ are the coefficients of the vector $X$ with respect to the basis $u$.

Consider in $L(N)$ the Sasaki--Mok metric \cite{mok,ks0,cl} 
\begin{align*}
g_{SM}(X^h_u,Y^h_u) &=g(X,Y),\\
g_{SM}(X^h_u,Y^{v,j}) &=0, \\
g_{SM}(X^{v,i},Y^{v,j}) &=\delta_{ij}g(X,Y).
\end{align*}
This implies that the decomposition $T_uL(N)=\mathcal{H}_u\oplus\mathcal{V}^1_u\oplus\ldots\oplus\mathcal{V}^n_u$ is orthogonal. 

For a tensor field $T:TN\to TN$ put 
\begin{equation*}
\bar{T}_u=\sum_i T(u_i)^{v,i}_u,\quad u\in L(N).
\end{equation*}
Then $\bar{T}$ is a vertical vector field on $L(N)$.
\begin{prop}\label{prop:verticalinvariant}
The correspondence
\begin{equation*}
T\mapsto \bar{T}
\end{equation*}
is a bijection between $(1,1)$--tensor fields on $N$ and invariant vertical vector fields on $L(N)$, i.e. vertical vector fields $V$ such that
\begin{equation*}
R_{g\ast}V_u=V_{ug},\quad u\in L(N), g\in GL(n).
\end{equation*}
\end{prop}
\begin{proof}
Let $T$ be $(1,1)$--tensor field. For $u\in L(N)$ put $T(u_i)=\sum_j t_i^j(u)u_j$. Then
\begin{equation*}
\bar{T}_u=\sum_{i,j}t_i^j(u)(E^j_i)^{\ast}_u.
\end{equation*}
The invariance of $\bar{T}$ follows from the relation
\begin{equation*}
t_i^j(ug)=\sum_{k,l}g^k_i\,\tilde{g}^j_k\, t^l_k(u),
\end{equation*}
where $(\tilde{g}^i_j)$ is the matrix inverse to $(g^i_j)$, and the well known formula $R_{g\ast} A^{\ast}_u=({\rm ad}(g^{-1})A)^{\ast}_{ua}$ for any $A\in gl(n+p)$.

Conversely, if $V\in\mathcal{V}$ is invariant vertical vector field, then for fixed basis $u\in L(N)$ we have
\begin{equation*}
V_u=\sum_i (v_i)^{v,i}_u,
\end{equation*}
for some vectors $v_i$. It suffices to put $T(u_i)=v_i$. The fact that $T$ is well defined follows by the invariance of $V$. 
\end{proof}

\begin{rem}
Notice that the vector field $\bar{T}$ can be written in the form
\begin{equation*}
\bar{T}_u=(u^{-1}\cdot T\cdot u)^{\ast}_u,
\end{equation*}
where we treat the basis $u\in L(N)$ as a linear isomorphism $u:\mathbb{R}^{n+p}\to T_{\pi(u)}N$. Then the invariance of $T$ follows immediately by the correspondence $R_{g\ast}A^{\ast}_u=({\rm ad}(g^{-1})A)^{\ast}_{ug}$.
\end{rem}

\section{Adapted orthonormal frames}

Let $M$ be a $p$--dimensional submanifold of $(p+n)$--dimensional Riemannian manifold $(N,g)$ and denote by $O(N)$ the bundle of orthonormal frames over $N$. We say that a frame $u=(u_1,\ldots,u_{n+p})\in O(N)$ over the point $x\in M$ is {\it adapted} if the first $p$ vectors $u_1,\ldots,u_p$ are tangent to $M$. Then, clearly, the remaining vectors $u_{p+1},\ldots,u_{p+n}$ are orthogonal to $M$. Denote the subbundle of adapted frames by $O(M,N)$. We have the following diagram
\begin{equation*}
\begin{diagram}
\node{O(M,N)}\arrow{e,J}\arrow{se}\node{O(N)|M}\arrow{e,J}\arrow{s}\node{O(N)}\arrow{s}\\
\node[2]{M}\arrow{e,J}\node{N}
\end{diagram}
\end{equation*}
where the horizontal arrows are just inclusions. The structure group of $O(M)$ or $O(M)|N$ is $G=O(p+n)$, whereas the structure group of $O(M,N)$ is $H=O(p)\oplus O(n)\subset O(p+n)$.

Let $\nabla$ be the Levi--Civita connection on $N$ and let $\omega$ denote its connection form on $O(M)$. $\omega$ has its values in the Lie algebra $\mathfrak{g}=o(p+n)$. The $\mathfrak{h}=o(n)\oplus o(p)$--component $\omega '$ of $\omega$ restricted to $M$ defines the connection in $O(M,N)$. Moreover, the $o(n)$--component of $\omega '$ defines the Levi--Civita connection in $O(M)=O(M,N)/O(p)$ \cite{kn}.

Notice that tangent bundle $TN$ over $M$ is an associated bundle with $O(M,N)$, i.e.
\begin{equation*}
TN|M=O(M,N)\times_H\mathbb{R}^{p+n}.
\end{equation*}
In addition, we have
\begin{equation*}
TM=O(M,N)\times_{O(p)}\mathbb{R}^p,\quad T^{\bot}M=O(M,N)\times_{O(n)}\mathbb{R}^n,
\end{equation*}
where $\mathbb{R}^p,\mathbb{R}^n$ are considered as subspaces of $\mathbb{R}^{p+n}$ as restrictions to first $p$ and last $n$ coordinates, respectively. Denote by $\nabla '$ and $\nabla^M$ the connections on $TN|M$ and $TM$, respectively, induced by connection form $\omega '$. Then
\begin{align*}
\nabla '_XY &=(\nabla_X Y^{\top})^{\top}+(\nabla_X Y^{\bot})^{\bot},\\
\nabla^M_XZ &=\left(\nabla_X Z\right)^{\top}
\end{align*}
for $X\in TM$, $Y\in\Gamma(TN|M)$ and $Z\in\Gamma(TM)$. Notice that restrictions of $\nabla$ and $\nabla '$ to $TM$ define the same connection $\nabla^M$. Therefore, we will often write $\nabla '$ instead of $\nabla^M$. 

Denote by $S$ the $(1,2)$--tensor being the difference between $\nabla$ and $\nabla '$, i.e.,
\begin{equation*}
S_XY=(\nabla_XY^{\top})^{\bot}+(\nabla_XY^{\bot})^{\top},\quad X\in TM,\quad Y\in TN|M.
\end{equation*} 

If we denote by $\Pi$ the second fundamental form of $M$ and by $A_N$ the Weingarten operator of $M$ with respect to normal vector field $N$, then tensor $S$ can be written in the form
\begin{equation*}
S_XY=\Pi(X,Y)-A_{Y^{\bot}}(X),
\end{equation*}
which implies that $S$ vanishes if and only if $M$ is totally geodesic.

Let us now compute the horizontal distribution of $O(M,N)$ with respect to $\omega '$.

\begin{prop}\label{prop:lifttoomn}
The horizontal lift $X^{h'}$ of a vector $X\in TM$ to the bundle $O(M,N)$ equals
\begin{equation}\label{eq:horliftomn}
\begin{split}
X^{h'}_u &=X^h_u+(\overline{S_X})_u \\
&=X^h_u+\sum_A \Pi(X,u_A)^{v,A}_u+\sum_{\alpha} A_{u_{\alpha}}(X)^{v,\alpha}_u,
\end{split}
\end{equation}
where $X^h_u$ denotes the horizontal lift of $X$ to $O(N)|M$.
\end{prop}
\begin{proof}
Fix a vector $X\in TM$ and a basis $u\in O(M,N)$. Let $f_i:O(N)_u\to\mathbb{R}^n$ denotes the invariant function such that $f_i(u)=e_i$, where $e_i$ is the $i$--th vector of the canonical basis in $\mathbb{R}^n$. We will show that for a vertical vector $V\in\mathcal{V}_u$ we have
\begin{equation*}
V_u=-\sum_i(u(Vf_i))^{v,i}_u.
\end{equation*}  
Indeed, since $A^{\ast}_u f_i=-A\cdot f_i(u)=-Ae_i$ is the $i$--th column of $-A$ we get
\begin{equation*}
-\sum_i (u(A^{\ast}_u f_i))^{v,i}_u=\sum_i (u(Ae_i))^{v,i}_u=A^{\ast}_u.
\end{equation*}
Consider the vertical vector $V=X^h_u-X^{h'}_u$. Then by above
\begin{align*}
(\overline{S_X})_u=\sum_i S_X(u_i)^{v,i}_u &=\sum_i\left(\nabla_X u_i-\nabla '_X u_i\right)^{v,i}_u=\sum_i (u(Vf_i))^{v,i}=-V,
\end{align*}
which proves \eqref{eq:horliftomn}.
\end{proof}

\begin{rem}
It follows by the proof of Proposition \ref{prop:lifttoomn} that if $\nabla$ and $\nabla '$ are two connections on $N$, $S=\nabla-\nabla '$ is their difference, then the relation between horizontal lifts to $L(M)$ equals $X^{h'}=X^h+\overline{S_X}$.
\end{rem}

\section{Properties of some operators and vector fields}

In this section, we will introduce necessary operators and vector fields which will play very important role. These operators allow to decompose vectors in the frame bundle $O(N)$ into tangent and orthogonal parts to $O(M,N)$.

Denote by $\mathfrak{g}(TN)$ the space of skew--symmetric $(1,1)$--tensor fields on $TN$, i.e. $T\in\mathfrak{g}(TN)$ if and only if
\begin{equation*}
g(T(X),Y)=-g(X,T(Y)),\quad X,Y\in TN.
\end{equation*}
The following Lemma will be used frequently.

\begin{lem}\label{lem:lifttensor}
Let $T\in\mathfrak{g}(TN)$. Then the vector field $\bar{T}|O(N)$ is tangent to $O(N)$. In particular, vector fields
\begin{equation*}
\overline{S_X}\quad\textrm{and}\quad\overline{R(X,Y)},\quad X,Y\in TM,
\end{equation*} 
are tangent to $O(N)$.
\end{lem} 
\begin{proof}
Let $T(u_i)=\sum_j t^j_i u_j$. Assuming $T$ is skew--symmetric we have $t^j_i=-t^i_j$, hence
\begin{equation*}
\bar{T}_u=\sum_{i,j}t^j_i(E^j_i)^{\ast}_u=\sum_{i<j}t^j_i(E^j_i-E^i_j)^{\ast}_u.
\end{equation*}
Therefore $\bar{T}_u\in T_uO(N)$.
\end{proof}

Denote by $\mathfrak{m}$ the natural complement of the Lie algebra $\mathfrak{h}$ in the Lie algebra of skew--symmetric matrices $\mathfrak{g}$. This decomposition allows to decompose skew--symmetric tensors over $M$ as follows. For $T\in\mathfrak{g}(TN|M)$ put
\begin{align*}
T_{\mathfrak{m}}(X) &=T(X^{\bot})^{\top}+T(X^{\top})^{\bot}, \\
T_{\mathfrak{h}}(X) &=T(X^{\bot})^{\bot}+T(X^{\top})^{\top}.
\end{align*}
Clearly $T=T_{\mathfrak{h}}+T_{\mathfrak{m}}$ and $T_{\mathfrak{h}}$, $T_{\mathfrak{m}}$ are skew--symmetric. Denote the subspaces spanned by these tensors by $\mathfrak{h}(TN|M)$ and $\mathfrak{m}(TN|M)$, respectively.

\begin{lem}\label{lem:tensordecomp}
The decomposition
\begin{equation}\label{eq:tensordecomp}
\mathfrak{g}(TN|M)=\mathfrak{h}(TN|M)\oplus\mathfrak{m}(TN|M)
\end{equation} 
implies the following geometric conditions.
\begin{enumerate}
\item For $T\in\mathfrak{g}(TN|M)$ we have
\begin{align*}
(\nabla_X T_{\mathfrak{h}})_{\mathfrak{m}} &=[S_X,T_{\mathfrak{h}}], 
&&(\nabla_X T_{\mathfrak{h}})_{\mathfrak{h}}=\nabla '_XT_{\mathfrak{h}},\\
(\nabla_X T_{\mathfrak{m}})_{\mathfrak{m}} &=\nabla '_XT_{\mathfrak{m}},
&&(\nabla_X T_{\mathfrak{m}})_{\mathfrak{h}}=[S_X,T_{\mathfrak{m}}].
\end{align*}
In particular, $\nabla '$ obeys the decomposition \eqref{eq:tensordecomp}.
\item The following Codazzi and Gauss formulas hold
\begin{align*}
R(X,Y)_{\mathfrak{h}} &=R'(X,Y)+[S_X,S_Y],\\
R(X,Y)_{\mathfrak{m}} &=\nabla '_X S_Y-\nabla '_Y S_X-S_{[X,Y]}.
\end{align*}
\end{enumerate}
\end{lem}
\begin{proof}
Follows by simple computations. We will just compute one formula. We have
\begin{align*}
(\nabla_X T_{\mathfrak{h}})_{\mathfrak{m}}(Z) &=(\nabla_X T_{\mathfrak{h}})(Z^{\top})^{\bot}+(\nabla_X T_{\mathfrak{h}})(Z^{\bot})^{\top}\\
&=(\nabla_X T_{\mathfrak{h}}(Z^{\top})-T_{\mathfrak{h}}(\nabla_X Z^{\top}))^{\bot}+(\nabla_X T_{\mathfrak{h}}(Z^{\bot})-T_{\mathfrak{h}}(\nabla_X Z^{\bot}))^{\top}\\
&=(\nabla_X (T(Z^{\top}))^{\top})^{\bot}-(T((\nabla_X Z^{\top})^{\bot}))^{\bot}\\
&+(\nabla_X (T(Z^{\bot}))^{\bot})^{\top}-((T(\nabla_X Z^{\bot})^{\top}))^{\top}\\
&=S_X((T(Z^{\top}))^{\top})-(T(S_X(Z^{\top})))^{\bot}+S_X((T(Z^{\bot}))^{\bot})-(T(S_X(Z^{\bot})))^{\top}\\
&=S_X(T_{\mathfrak{h}}(Z))-T_{\mathfrak{h}}(S_X(Z))\\
&=[S_X,T_{\mathfrak{h}}](Z).
\end{align*}
\end{proof}

Now, for $T\in\mathfrak{g}(TN|M)$ define
\begin{align*} 
R_T(X) &=\sum_i R(e_i,T(e_i))X,\\
S_{T_{\mathfrak{m}}} &=\sum_A S_{e_A}(T_{\mathfrak{m}}(e_A))-\sum_{\alpha}S_{T_{\mathfrak{m}}(e_{\alpha})}(e_{\alpha}),
\end{align*}
where $X\in TN|M$. Notice that, by the symmetry of the second fundamental form $\Pi$ of $M$,
\begin{equation*}
S_{T_{\mathfrak{m}}}=2\sum_A S_{e_A}(T_{\mathfrak{m}}(e_A)),
\end{equation*}
which implies, in particular, that
\begin{equation*}
S_{S_X}=2\sum_A S_{e_A}^2(X),\quad X\in TM.
\end{equation*}

Let us collect the properties of vector $S_{T_{\mathfrak{m}}}$ and operators $S_X$, $R_T$. Denote by $\skal{T}{T'}$ the inner product of skew--symmetric endomorphism equal
\begin{equation*}
\skal{T}{T'}=\sum_i g(T(e_i),T'(e_i))=-{\rm tr}(T\cdot T'),
\end{equation*}
where $(e_i)$ is any orthonormal basis. Notice, that $T_{\mathfrak{m}}$ and $T'_{\mathfrak{h}}$ are orthogonal and for any skew--symmetric endomorphisms $T,T',T''$
\begin{equation*}
\skal{[T,T']}{T''}=\skal{T}{[T',T'']},
\end{equation*}
where $[T,T']$ denotes the commutator of $T$ and $T'$.

\begin{lem}\label{lem:propoperators}
Let $X,Y\in TM$ and $T\in\mathfrak{g}(TN|M)$. Then, the following relations hold
\begin{align*}
& g(R_T(X),Y)=\skal{R(X,Y)}{T},\\
& g(S_{T_\mathfrak{m}},X)=-\skal{T_{\mathfrak{m}}}{S_X}.
\end{align*}
\end{lem}
\begin{proof}
For the proof of the first relation, we have
\begin{align*}
g(R_T(X),Y) &=\sum_i g(R(u_i,T(u_i))X,Y)=\sum_i g(R(X,Y)u_i,T(u_i))\\
&=\skal{R(X,Y)}{T}.
\end{align*}
Further
\begin{align*}
g(S_{T_{\mathfrak{m}}},X) &=\sum_A g(S_{e_A}(T_{\mathfrak{m}}(e_A)),X)-\sum_{\alpha}g(S_{T_{\mathfrak{m}}(e_{\alpha})}(e_{\alpha}),X)\\
&=-\sum_A g(T_{\mathfrak{m}}(e_A),S_{e_A}(X))+\sum_{\alpha}g(e_{\alpha},S_{T_{\mathfrak{m}}(e_{\alpha})}(X))\\
&=-\sum_A g(T_{\mathfrak{m}}(e_A),S_X(e_A))-\sum_{\alpha}g(S_X(e_{\alpha}),T_{\mathfrak{m}}(e_{\alpha}))\\
&=-\skal{T_{\mathfrak{m}}}{S_X}.
\end{align*}
\end{proof}

For any $X\in TM$ put
\begin{align*}
P(X)=X-2\sum_A S_{e_A}^2(X)=X-S_{S_X},
\end{align*} 
where $(e_A)$ is any orthonormal basis on $M$.

\begin{lem}\label{lem:tildeg}
Operator $P$ satisfies the following relation
\begin{equation*}
g_{SM}(X^{h'},Y^{h'})=g(X,P(Y)),\quad X,Y\in TM.
\end{equation*}
In particular, $P:TM\to TM$ is symmetric and invertible $(1,1)$--tensor field. Moreover,
\begin{equation*}
g((\nabla '_X P)Y,Z)=\skal{S_Z}{\nabla '_XS_Y-S_{\nabla '_XY}}+\skal{S_Y}{\nabla '_XS_Z-S_{\nabla '_XZ}}
\end{equation*}
\end{lem}
\begin{proof}
Notice that for any $X,Y\in TM$ and $u\in O(M,N)$ by Lemma \ref{lem:propoperators}
\begin{align*}
g(X,P(Y)) &=g(X,Y-S_{S_Y})\\
&=g(X,Y)+\skal{S_X}{S_Y}\\
&=g(X,Y)+\sum_i g(S_X(u_i),S_Y(u_i))\\
&=g(X,Y)+g_{SM}(\overline{S_X},\overline{S_Y})\\
&=g_{SM}(X^{h'},Y^{h'}).
\end{align*}
Since $|X^{h'}|_{G_{SM}}^2=g(X,P(X))$, it follows that $P$ is invertible. Moreover,
\begin{align*}
g((\nabla '_X P)Y,Z) &=g(\nabla '_X(Y-S_{S_Y}),Z)-g(\nabla '_X Y,Z-S_{S_Z})\\
&=-g(\nabla '_X S_{S_Y},Z)+g(\nabla '_XY,S_{S_Z})\\
&=X\skal{S_Y}{S_Z}+g(S_{S_Y},\nabla '_XZ)-\skal{S_{\nabla '_XY}}{S_Z}\\
&=\skal{\nabla '_XS_Y}{S_Z}+\skal{S_Y}{\nabla '_XS_Z}-\skal{S_Y}{S_{\nabla '_XZ}}-\skal{S_{\nabla '_XY}}{S_Z},
\end{align*}
which proves the second part of the lemma.
\end{proof}

By above Lemma, the symmetric and bilinear form 
\begin{equation*}
\tilde{g}(X,Y)=g(X,P(Y)),\quad X,Y\in TM, 
\end{equation*}
defines the Riemannian metric on $M$. Let $\tilde{\nabla}$ be the Levi--Civita connection of $\tilde{g}$ on $M$. One can show, see for example \cite{gm}, that
\begin{multline}\label{eq:GilMedranoeq}
g(\tilde{\nabla}_X Y-\nabla '_XY,P(Z))\\
=\frac{1}{2}\left(g((\nabla '_XP)Y,Z)+g((\nabla '_Y P)X,Z)-g(X,(\nabla '_ZP)X)\right) 
\end{multline}
for any $X,Y,Z\in TM$. Denote the difference of connections $\tilde{\nabla}$ and $\nabla '$ by $L$, i.e.
\begin{equation*}
L_XY=\tilde{\nabla}_XY-\nabla '_XY,\quad X,Y\in TM.
\end{equation*}
We will describe tensor $L$ in a different way. Put
\begin{equation*}
Q_T(X)=P^{-1}((R_T(X))^{\top}-S_{(\nabla_XT)_{\mathfrak{m}}}),\quad X\in TM.
\end{equation*}

\begin{lem}\label{lem:propofQT}
The operator $Q_T$ has the following properties
\begin{align*}
g_{SM}(Q_{T_{\mathfrak{h}}}(X)^{h'},Y^{h'}) &=\skal{R'(X,Y)}{T_{\mathfrak{h}}},\\
g_{SM}(Q_{T_{\mathfrak{m}}}(X)^{h'},Y^{h'}) &=\skal{\nabla '_XS_Y-\nabla '_YS_X-S_{[X,Y]}}{T_{\mathfrak{m}}}+\skal{\nabla '_X T_{\mathfrak{m}}}{S_Y}.
\end{align*}
\end{lem}
\begin{proof}
Let $X,Y\in TM$ and $T\in\mathfrak{g}(TN|M)$. Then, by Lemma \ref{lem:propoperators} and Lemma \ref{lem:tensordecomp},
\begin{align*}
g_{SM}(Q_{T_{\mathfrak{h}}}(X)^{h'},Y^{h'}) &=g(P(Q_{T_{\mathfrak{h}}}(X)),Y)=g(R_{T_{\mathfrak{h}}}(X)^{\top}-S_{(\nabla_X T_{\mathfrak{h}})_{\mathfrak{m}}},Y)\\
&=\skal{R(X,Y)}{T_{\mathfrak{h}}}+\skal{(\nabla_XT_{\mathfrak{h}})_{\mathfrak{m}}}{S_Y}\\
&=\skal{R'(X,Y)}{T_{\mathfrak{h}}}+\skal{[S_X,S_Y]}{T_{\mathfrak{h}}}+\skal{[S_X,T_{\mathfrak{h}}]}{S_Y}\\
&=\skal{R'(X,Y)}{T_{\mathfrak{h}}}.
\end{align*}
and, as above,
\begin{align*}
g_{SM}(Q_{T_{\mathfrak{m}}}(X)^{h'},Y^{h'}) &=\skal{R(X,Y)}{T_{\mathfrak{m}}}+\skal{(\nabla_XT_{\mathfrak{m}})_{\mathfrak{m}}}{S_Y}\\
&=\skal{\nabla '_XS_Y-\nabla '_YS_X-S_{[X,Y]}}{T_{\mathfrak{m}}}+\skal{\nabla '_X T_{\mathfrak{m}}}{S_Y}.
\end{align*}
\end{proof}

Now, we can state and prove the most important result of this section.
\begin{prop}\label{prop:operatorL}
The operator $L$ can be described as follows
\begin{equation*}
L_XY=\frac{1}{2}\left( Q_{S_X}(Y)+Q_{S_Y}(X)+P^{-1}(S_{S_{\nabla '_XY+\nabla '_YX}}) \right).
\end{equation*}
\end{prop}
\begin{proof}
Denote the right hand side by ${\rm RHS}$. By Lemma \ref{lem:propofQT} we have
\begin{align*}
2g({\rm RHS},P(Z)) &=\skal{\nabla '_Y S_Z-\nabla '_ZS_Y-S_{[Y,Z]}}{S_X}+\skal{\nabla '_YS_X}{S_Z}\\
&+\skal{\nabla '_XS_Z-\nabla '_ZS_X-S_{[X,Z]}}{S_Y}+\skal{\nabla '_XS_Y}{S_Z}\\
&-\skal{S_{\nabla '_XY+\nabla '_YX}}{S_Z}\\
&=\skal{\nabla '_YS_Z-S_{\nabla '_YZ}}{S_X}+\skal{\nabla '_YS_X-S_{\nabla '_YX}}{S_Z}\\
&+\skal{\nabla '_XS_Z-S_{\nabla '_XZ}}{S_Y}+\skal{\nabla '_XS_Y-S_{\nabla '_XY}}{S_Z}\\
&-(\skal{\nabla '_ZS_Y-S_{\nabla '_ZY}}{S_X}+\skal{\nabla '_ZS_X-S_{\nabla '_ZX}}{S_Y}).
\end{align*}
Therefore, by Lemma \ref{lem:tildeg} and equation \eqref{eq:GilMedranoeq} we get desired equality.
\end{proof}

\section{Geometry of adapted orthonormal frame bundle induced by a submanifold}

Consider the frame bundle $L(N)$ over the Riemannian manifold $N$ and equip $L(N)$ with Sasaki--Mok metric. Then the Levi--Civita connection $\nabla^{L(N)}$ at $u\in L(N)$ is of the form \cite{cl,ks0}
\begin{align*}
\nabla^{L(N)}_{X^h}Y^h &=\left(\nabla_X Y\right)^h-\frac{1}{2}\overline{R(X,Y)},\\
\nabla^{L(N)}_{X^h}Y^{v,j} &=\left(\nabla_X Y\right)^{v,j}+\frac{1}{2}\left( R(u_j,Y)X \right)^h,\\
\nabla^{L(N)}_{X^{v,i}}Y^h &=\frac{1}{2}\left( R(u_j,X)Y \right)^h,\\
\nabla^{L(N)}_{X^{v,i}}Y^{v,j} &=0.
\end{align*}

By Lemma \ref{lem:lifttensor} the vector $\nabla^{L(N)}_{X^h}Y^h$ is tangent to $O(N)$ at the elements $u\in O(N)$. Moreover, by the definition of a vertical vector field $\bar{T}$ induced by a tensor field $T$, we get
\begin{align}
\nabla^{L(N)}_{\bar{T}}X^h &=\frac{1}{2}\left( \sum_i R(u_i,T(u_i))X \right)^{h}, \notag \\
\nabla^{L(N)}_{X^h}\bar{T} &=\frac{1}{2}\left( \sum_i R(u_i,T(u_i))X \right)^{h}+\overline{\nabla_X T}, \label{eq:nablaTbar}\\
\nabla^{L(N)}_{\bar{T}}\bar{T'} &=\overline{T'\circ T}. \notag
\end{align}
It follows by Lemma \ref{lem:lifttensor} that at $u\in O(N)$ the first two relations of \eqref{eq:nablaTbar} are relations for the Levi--Civita connection $\nabla^{O(N)}$ on $O(N)$. Since for $T$ and $T'$ the decomposition
\begin{equation*}
T'\circ T=\frac{1}{2}[T',T]+\frac{1}{2}\left( T'\circ T+T\circ T' \right),
\end{equation*} 
where $[T',T]$ is a commutator of $T'$ and $T$, is orthogonal and $[T',T]\in\mathfrak{g}(TN)$, again by Lemma \ref{lem:lifttensor}, we get $\nabla^{O(N)}_{\bar{T}}\bar{T}'=\frac{1}{2}\overline{[T',T]}$. Moreover, vector $\sum_i R(u_i,T(u_i))X$ is independent of the choice of the orthonormal basis $u$. Finally
\begin{equation}\label{eq:nablaON}
\begin{split}
\nabla^{O(N)}_{X^h}Y^h &=\left(\nabla_X Y\right)^h-\frac{1}{2}\overline{R(X,Y)},\\
\nabla^{O(N)}_{\bar{T}}X^h &=\frac{1}{2}R_T(X)^h, \\
\nabla^{O(N)}_{X^h}\bar{T} &=\frac{1}{2}R_T(X)^h+\overline{\nabla_X T}, \\
\nabla^{O(N)}_{\bar{T}}\bar{T'} &=\frac{1}{2}\overline{[T',T]}.
\end{split}
\end{equation}
The Levi--Civita connection of $O(N)$ was first computed by Kowalski and Sekizawa \cite{ks1,ks2}, but with the use of different vertical vector fields, namely the vector fields $(T_{ij})^{\ast}$ defined below.

\begin{rem}\label{rem:antiisom}
By the formula \eqref{eq:nablaTbar} for the Levi--Civita connection $\nabla^{L(N)}$ on the vector fields $\bar{T},\bar{T'}$ we get that
\begin{equation*}
[\bar{T},\bar{T}']=-\overline{[T,T']}.
\end{equation*}
Therefore, the subspace of invariant vertical vector fields forms a Lie algebra, which is anti--isomorphic to the Lie algebra of $(1,1)$--tensor fields of the tangent space $TN$. Moreover, we get
\begin{equation*}
\nabla^{O(N)}_{\bar{T}}\bar{T'}=\frac{1}{2}[\bar{T},\bar{T'}].
\end{equation*}
Hence, the Levi--Civita connection on the invariant vertical vector fields is just the bi--invariant connection.
\end{rem}

Put
\begin{equation*}
T_{ij}=\frac{1}{\sqrt{2}}\left(E^i_j-E^j_i\right)\in o(n).
\end{equation*}
Then, by the equality $(E^i_j)^{\ast}_u=u_i^{v,j}$, we have
\begin{equation*}
g_{SM}((T_{ij})^*_u,(T_{kl})^*_u)=\delta_{ik}\delta_{jl}-\delta_{il}\delta_{jk}.
\end{equation*}
Therefore 
\begin{equation*}
g_{SM}(A^{\ast}_u,B^{\ast}_u)=-{\rm tr}(A\cdot B),
\end{equation*}
hence Sasaki--Mok metric $g_{SM}$ is on the vertical subspace of the orthonormal frame bundle $O(N)$ the metric induced by the Killing form $K$ on the Lie algebra $\mathfrak{g}$,
\begin{equation*}
K(U,V)=-{\rm tr}(U\cdot V),\quad U,V\in\mathfrak{g}.
\end{equation*}

Let $M$ be a submanifold of $N$ and $O(M,N)$ the bundle of frames adapted to $M$. Notice that vectors 
\begin{equation*}
X^{h'},\quad (T_{AB})^{\ast}, (T_{\alpha,\beta})^{\ast},
\end{equation*}
where $X\in TM$, are tangent to $O(M,N)$, whereas vectors
\begin{equation*}
Z^{h},\quad (T_{A\alpha})^{\ast}_u-(S_{u_A}(u_{\alpha}))^h_u,
\end{equation*}
where $Z\in T^{\bot}M$, are orthogonal to $O(M,N)$ in $O(N)$.

Denote by $\nabla^{O(M,N)}$ the Levi--Civita connection of $O(M,N)$ and by $\Pi^{O(M,N)}$ the second fundamental form of $O(M,N)$ in $O(N)$ with respect to Sasaki--Mok metric $g_{SM}$.

Adopt the notation from the previous sections.

\begin{prop}\label{prop:horizontaldecomp}
Let $X\in TM$ and $u\in O(M,N)$. The decomposition of $X^h_u$ into tangent and orthogonal part to $O(M,N)$ in $O(N)$ equals
\begin{equation}\label{eq:decompositionXh}
X^h_u=P^{-1}(X)^{h'}_u+\left(-\overline{S_{P^{-1}(X)}}-2\sum_A (S_{e_A}^2(P^{-1}(X)))^h_u\right).
\end{equation}
\end{prop}
\begin{proof}
Since $X^{h'}=X^h_u+\overline{S_X}$, it follows that
\begin{equation}\label{eq:PXh}
X^{h'}_u+\Big(-\overline{S_X}-2\sum_A (S_{e_A}^2(X))^h_u\Big)=
X^h_u-2\sum_A \left(S_{e_A}^2(X)\right)^h_u=P(X)^h_u.
\end{equation}

Put $V_u=\overline{S_X}+2\sum_A (S_{e_A}^2(X))^h_u$. Then, by above,
\begin{align*}
g_{SM}(Y^{h'},V_u) &=
g_{SM}(\overline{S_X},\overline{S_Y})+g(Y,2\sum_A S_{e_A}^2(X))\\
&=2\sum_Ag(S_{u_A}(X),S_{u_A}(Y))-2\sum_Ag(S_{u_A}(X),S_{u_A}(Y))\\
&=0.
\end{align*}
Therefore the vector $V_u$ is orthogonal to subspace $\mathcal{H}'$. It is easy to see that $V_u$ is also orthogonal to the vertical subspace of $O(N,M)|M$. Thus $V_u\in T^{\bot}O(M,N)$. This implies that \eqref{eq:PXh} is a decomposition into tangent and orthogonal parts. Substituting $X$ by $P^{-1}(X)$ in \eqref{eq:PXh} we get \eqref{eq:decompositionXh}.
\end{proof}

\begin{prop}\label{prop:veritdecomp}
Let $T\in\mathfrak{g}(TN|M)$. The following decomposition 
\begin{multline}\label{eq:vertdecomp}
\bar{T}=\left(\bar{T_{\mathfrak{h}}}-P^{-1}(S_{T_{\mathfrak{m}}})^{h'}\right)+\\
\left( \bar{T_{\mathfrak{m}}}+(S_{T_{\mathfrak{m}}})^h+\overline{S_{P^{-1}(S_{T_{\mathfrak{m}}})}}+2\sum_A (S_{e_A}^2(P^{-1}(S_{T_{\mathfrak{m}}})))^h\right)
\end{multline}
is a decomposition into tangent and orthogonal part to $O(M,N)$ in $O(N)$. 
\end{prop}
\begin{proof}
By the definition of Sasaki--Mok metric we have
\begin{equation*}
g_{SM}(\bar{T_{\mathfrak{h}}},\bar{T_{\mathfrak{m}}})=\skal{T_{\mathfrak{h}}}{T_{\mathfrak{m}}}=0.
\end{equation*}
Therefore the decomposition
\begin{equation*}
\bar{T}=\bar{T_{\mathfrak{h}}}+\bar{T_{\mathfrak{h}}}
\end{equation*}
is orthogonal. Moreover vector $\bar{T_{\mathfrak{m}}}+(S_{T_{\mathfrak{m}}})^h$ is orthogonal to vectors $T_{AB}^{\ast}$, $T_{\alpha\beta}^{\ast}$ and for any $X\in TM$ by Lemma \ref{lem:propoperators}
\begin{align*}
g_{SM}(\bar{T_{\mathfrak{m}}}+(S_{T_{\mathfrak{m}}})^h,X^{h'}) &=g_{SM}(\bar{T_{\mathfrak{m}}}+(S_{T_{\mathfrak{m}}})^h,X^h+\overline{S_X})\\
&=\skal{T_{\mathfrak{m}}}{S_X}+g(S_{T_{\mathfrak{m}}},X)\\
&=0
\end{align*}
Therefore $\bar{T_{\mathfrak{m}}}+(S_{T_{\mathfrak{m}}})^h\in T^{\bot}O(M,N)$. We have
\begin{equation*}
\bar{T}=\bar{T_{\mathfrak{h}}}-(S_{T_{\mathfrak{m}}})^h+\bar{T_{\mathfrak{m}}}+(S_{T_{\mathfrak{m}}})^h.
\end{equation*}
To prove \eqref{eq:vertdecomp} it suffices to apply Proposition \ref{prop:horizontaldecomp} to the vector $(-S_{T_{\mathfrak{m}}})^h$.
\end{proof}

We will now compute the Levi--Civita connection $\nabla^{O(N)}$ on the vector fields tangent to $O(M,N)$. By Lemma \ref{lem:tensordecomp} we have
\begin{equation}\label{eq:nablaONtangentvect}
\begin{split}
\nabla^{O(N)}_{X^{h'}}Y^{h'} &=(\nabla_XY)^h+\frac{1}{2}(R_{S_X}(Y)+R_{S_Y}(X))^h-\overline{R'(X,Y)}\\
&+\frac{1}{2}\overline{\nabla '_XS_Y+\nabla '_YS_X}+\frac{1}{2}\overline{S_{[X,Y]}}\\
\nabla^{O(N)}_{X^{h'}}\bar{T} &=\frac{1}{2}R_T(X)^h+\overline{\nabla '_XT}+\frac{1}{2}\overline{(\nabla_X T)_{\mathfrak{m}}}\\
\nabla^{O(N)}_{\bar{T}}Y^{h'} &=\frac{1}{2}R_T(Y)^h+\frac{1}{2}\overline{(\nabla_Y T)_{\mathfrak{m}}}.
\end{split}
\end{equation}

\subsection{Intrinsic geometry}

Now, we are ready to compute the Levi--Civita connection $\nabla^{O(M,N)}$ and curvatures of $O(M,N)$. 

\begin{prop}\label{prop:LeviCivitaOMN}
Let $X,Y\in\Gamma(TM)$ and $T,T'\in\mathfrak{h}(TN|M)$. The Levi--Civita connection $\nabla^{O(M,N)}$ of $O(M,N)$ equals
\begin{align*}
\nabla^{O(M,N)}_{X^{h'}}Y^{h'} &=(\tilde{\nabla}_XY)^{h'}-\frac{1}{2}\overline{R'(X,Y)},\\
\nabla^{O(M,N)}_{X^{h'}}\bar{T} &=\frac{1}{2}Q_T(X)^{h'}+\overline{\nabla '_X T},\\
\nabla^{O(M,N)}_{\bar{T}}Y^{h'} &=\frac{1}{2}Q_T(Y)^{h'},\\
\nabla^{O(M,N)}_{\bar{T}}\bar{T'} &=\frac{1}{2}[\bar{T},\bar{T}'].
\end{align*}
\end{prop}
\begin{proof}
Follows by Propositions \ref{prop:horizontaldecomp} and \ref{prop:veritdecomp}, Proposition \ref{prop:operatorL}, relations \eqref{eq:nablaONtangentvect} and the fact that
\begin{equation*}
P^{-1}(X)=X+P^{-1}(S_{S_X}),\quad X\in TM.
\end{equation*}
\end{proof}

\begin{prop}\label{prop:curvatureOMN}
Let $X,Y,Z\in TM$, $T,T',T''\in\mathfrak{h}(TN|M)$. The curvature tensor $R^{O(M,N)}$ equals
\begin{align*}
R^{O(M,N)}(X^{h'},Y^{h'})Z^{h'} &=(\tilde{R}(X,Y)Z)^{h'}-\frac{1}{2}\overline{(D_XR')(Y,Z)-(D_YR')(X,Z)}\\
&-\frac{1}{4}\left( Q_{R'(Y,Z)}(X)-Q_{R'(X,Z)}(Y)-2Q_{R'(X,Y)}(Z) \right)^{h'},\\
R^{O(M,N)}(X^{h'},Y^{h'})\bar{T} &=\frac{1}{2}\left( (D_XQ_T)(Y)-(D_YQ_T)(X) \right)^{h'}\\
&+\frac{1}{2}\overline{[R'(X,Y),T]}
-\frac{1}{4}\overline{R'(X,Q_T(Y))-R'(Y,Q_T(X))},\\
R^{O(M,N)}(X^{h'},\bar{T})Z^{h'} &=\frac{1}{2}((D_XQ_T)(Z))^{h'}
-\frac{1}{4}\overline{R'(X,Q_T(Z))-[R'(X,Z),T]}
\end{align*}
and
\begin{align*}
R^{O(M,N)}(X^{h'},\bar{T})\bar{T'} &=-\frac{1}{4}\left( Q_{[T,T']}(X)+Q_T(Q_{T'}(X) \right)^{h'},\\
R^{O(M,N)}(\bar{T},\bar{T'})Z^{h'} &=\left( \frac{1}{4}[Q_T,Q_{T'}](Z)+\frac{1}{2}Q_{[T,T']}(Z) \right)^{h'},\\
R^{O(M,N)}(\bar{T},\bar{T}')\bar{T''} &=-\frac{1}{4}\overline{[[T,T'],T'']},\\
\end{align*}
where
\begin{align*}
(D_XR')(Y,Z) &=\nabla '_X R'(Y,Z)-R'(\tilde{\nabla}_XY,Z)-R'(Y,\tilde{\nabla}_XZ),\\
(D_XQ_T)(Y) &=\tilde{\nabla}_XQ_T(Y)-Q_{\nabla '_XT}(Y)-Q_T(\tilde{\nabla}_XY).
\end{align*}
\end{prop}
\begin{proof}
Direct consequence of Proposition \ref{prop:LeviCivitaOMN}.
\end{proof}

\begin{prop}\label{prop:sectionalOMN}
The sectional curvature $\kappa^{O(M,N)}$ of $O(M,N)$ is of the form
\begin{align*}
\kappa^{O(M,N)}(X^{h'},Y^{h'}) &=\tilde{\kappa}(X,Y)-\frac{3}{4}\|R'(X,Y)\|^2,\\
\kappa^{O(M,N)}(X^{h'},\bar{T}) &=|Q_T^2(X)|^2_{\tilde{g}},\\
\kappa^{O(M,N)}(\bar{T},\bar{T'}) &=\frac{1}{8}\|[T,T']\|^2
\end{align*}
for $T,T'\in\mathfrak{h}(TN|M)$ orthonormal and $X,Y$ orthonormal with respect to $\tilde{g}$, where $\tilde{\kappa}(X,Y)$ is the sectional curvature of $(M,\tilde{g})$ in the direction of $X,Y\in TM$. 
\end{prop}
\begin{proof}
By Lemma \ref{lem:propofQT} we have
\begin{align*}
\kappa^{O(M,N)}(X^{h'},Y^{h'}) &=\tilde{g}(\tilde{R}(X,Y)Y,X)
+\frac{3}{4}g_{SM}(Q_{R'(X,Y)}(Y)^{h'},X^{h'})\\
&=\tilde{\kappa}(X,Y)-\frac{3}{4}\|R'(X,Y)\|^2.
\end{align*}
This proves the first equality. For the proof of the second relation, let $U_T(X)=R_T(X)^{\top}-S_{(\nabla_X T)_{\mathfrak{m}}}$ for $T\in\mathfrak{h}(TN|M)$. Then $Q_T(X)=P^{-1}(U_T(X))$. Moreover, since
\begin{equation*}
g(S_{(\nabla_X T)_{\mathfrak{m}}},Y)=-\skal{S_Y}{[S_X,T]}=\skal{[S_X,S_Y]}{T}=\skal{S_X}{[S_Y,T]}=-g(S_{(\nabla_Y T)_{\mathfrak{m}}},X),
\end{equation*} 
it follows that $U_T$ is skew-symmetric. Hence
\begin{equation*}
\tilde{g}(Q_T(X),Y)=g(U_T(X),Y)=-g(X,U_T(Y))=-\tilde{g}(X,Q_T(Y)),
\end{equation*}
so $Q_T$ is skew--symmetric with respect to $\tilde{g}$. Therefore
\begin{align*}
\kappa^{O(M,N)}(X^{h'},\bar{T}) &=-\frac{1}{4}g_{SM}(Q_T^2(X)^{h'},X^{h'})\\
&=-\frac{1}{4}\tilde{g}(Q_T^2(X),X)\\
&=\frac{1}{4}\tilde{g}(Q_T(X),Q_T(X)).
\end{align*}
The last relation follows directly by the formula for the curvature tensor.
\end{proof}

\begin{cor}
If $\dim M>2$, then $O(M,N)$ is never of constant sectional curvature.
\end{cor}

\begin{cor}
If $(M,g)$ is totally geodesic and $(N,g)$ is of constant sectional curvature $0\leq \kappa\leq\frac{2}{3}$, then $O(M,N)$ has non--negative sectional curvatures. 
\end{cor}
\begin{proof}
Since $M$ is totally geodesic, we have $S_X=0$ and $\tilde{g}=g$. In particular $\tilde{R}=R'$. Moreover
\begin{equation*}
\|R'(X,Y)\|^2=2\kappa^2.
\end{equation*}
By Proposition \ref{prop:sectionalOMN} we have $\kappa^{O(M,N)}(X^h,Y^h)=\kappa-\frac{3}{2}\kappa^2\geq 0$. Clearly the remaining sectional curvatures are non--negative.
\end{proof}

\subsection{Extrinsic geometry}

Now we are ready to compute the second fundamental form of the Levi--Civita connection of $g_{SM}$ and study the extrinsic geometry of $O(M,N)$. 

\begin{prop}
The second fundamental form $\Pi^{O(M,N)}$ equals
\begin{align*}
\Pi^{O(M,N)}(X^{h'}_u,Y^{h'}_u) &=\left( \Pi(X,Y)+\frac{1}{2}(R_{S_X}(Y)+R_{S_Y}(X))^{\bot} \right)^h\\
&-\frac{1}{2}\overline{S_{P^{-1}(\nabla '_XY+\nabla '_YX+(R_{S_X}(Y)+R_{S_Y}(X))^{\top})}}\\
&-\left(\sum_A S_{e_A}^2(P^{-1}(\nabla '_XY+\nabla '_YX+(R_{S_X}(Y)+R_{S_Y}(X))^{\top})) \right)^h\\
&+\frac{1}{2}\overline{\nabla '_XS_Y+\nabla '_YS_X}
+\frac{1}{2}(S_{\nabla '_XS_Y+\nabla '_YS_X})^h\\
&+\frac{1}{2}\overline{S_{P^{-1}(S_{\nabla '_XS_Y+\nabla '_YS_X})}}
+\left( \sum_A S_{e_A}^2(P^{-1}(S_{\nabla '_XS_Y+\nabla '_YS_X})) \right)^h,\\
\Pi^{O(M,N)}(X^{h'}_u,\bar{T}_u) &=\left( (R_T(X))^{\bot} \right)^h-\frac{1}{2}\overline{S_{P^{-1}(R_T(X)^{\top})}}-\left( \sum_A S_{e_A}^2(P^{-1}(R_T(X))^{\top}) \right)^h\\
&+\frac{1}{2}\overline{(\nabla_X T)_{\mathfrak{m}}}+\frac{1}{2}\overline{S_{P^{-1}(S_{(\nabla_X T)_{\mathfrak{m}}})}}
+\left(\sum_A S_{e_A}^2(P^{-1}(S_{(\nabla_X T)_{\mathfrak{m}}})) \right)^h,\\
\Pi^{O(M,N)}(\bar{T}_u,\bar{T'}_u) &=0,
\end{align*}
where $X,Y\in TM$, $T,T'\in\mathfrak{h}(TN|M)$ and $\Pi$ denotes the second fundamental form of $(M,g)$.
\end{prop}
\begin{proof}
First notice that
\begin{align*}
(\nabla '_X Y)^h+\frac{1}{2}\overline{S_{[X,Y]}} &=\frac{1}{2}(\nabla '_XY+\nabla '_YX)^h+\frac{1}{2}[X,Y]+\frac{1}{2}\overline{S_{[X,Y]}}\\
&=\frac{1}{2}(\nabla '_XY+\nabla '_YX)^h+\frac{1}{2}[X,Y]^{h'}.
\end{align*}
Now, Proposition follows by relations \eqref{eq:nablaONtangentvect} and Propositions \ref{prop:horizontaldecomp} and \ref{prop:veritdecomp}.  
\end{proof}

\begin{cor}\label{cor:totallygeodesic}
$O(M,N)$ is totally geodesic if and only if $M$ is totally geodesic and 
\begin{equation}\label{eq:totgeodR}
(R(U,V)W)^{\top}=0\quad\textrm{for $U,V,W\in T^{\bot}M$}.
\end{equation} 
\end{cor}
\begin{proof}
Assume $M$ is totally geodesic and \eqref{eq:totgeodR} holds. Then $\Pi=0$, $S_X=0$ and $R_T(X)^{\bot}=0$ for any $T\in\mathfrak{h}(TN|M)$ and $X\in TM$. Therefore $\Pi^{O(M,N)}=0$.

Conversely, assume $\Pi^{O(M,N)}=0$. Then by the formula for $\Pi^{O(M,N)}(X^{h'},\bar{T})$ it follows that $R_T(X)^{\bot}=0$. This implies \eqref{eq:totgeodR}. Moreover, by the formula for $\Pi^{O(M,N)}(X^{h'},Y^{h'})$ we get $\Pi=0$.
\end{proof}

Above corollary implies the following result.
\begin{cor}\label{cor:totgeodconst}
Assume $N$ is of constant scalar curvature. Then $O(M,N)$ is totally geodesic if and only if $M$ is totally geodesic.
\end{cor}

In order to compute the mean curvature vector $H^{O(M,N)}$ of $O(M,N)$ in $O(N)$ we need to recall the concept of the Gauss map of a submanifold and its harmonicity.

Let $Gr_n(N)$ denotes the Grassmann bundle of $n$--dimensional subspaces of the tangent bundle $TN$. Then
\begin{equation*}
Gr_n(N)=O(N)\times_G\left( G/H \right)=O(N)/H,
\end{equation*}
where the action of $G$ on the fiber is by left multiplication. The Killing form on $G$ induces the invariant metric on the quotient $G/H$, hence on the vertical subspace of the Grassmann bundle. Consider the mapping $\rho:O(N)\to Gr_n(N)$ of the form
\begin{equation*}
\rho(u)={\rm span}\{u_1,\ldots,u_n\},\quad u\in O(N).
\end{equation*}
In other words $\rho(u)=[u\cdot H]$, where $[u\cdot H]$ denotes the equivalence class of $u\cdot H\in O(N)\times (G/H)$ under the action of $G$. The horizontal subspace of $Gr_n(N)$ is defined as $\rho_{\ast}(\mathcal{H})$. The Riemannian metric $g_{Gr}$ on the horizontal space of the Grassmann bundle is induced from the metric $g$ on $M$. Moreover, horizontal and vertical subspaces are orthogonal. In other words, the map $\mu$ is a Riemannian submersion.

The Gauss map of a submanifold $M$ in $N$ is a map $\gamma:M\to Gr_n(N)$ defined as
\begin{equation*}
\gamma(x)=T_xM\subset T_x N,\quad x\in M.
\end{equation*}
In other words, the following diagram commutes
\begin{equation*}
\begin{diagram}
\node{O(N)|M}\arrow{e,t}{\rho}\arrow{se,b}{\pi}\node{Gr_n(N)|M}\arrow{e,J}\node{Gr_n(N)}\arrow{s,l}{\pi_{Gr}}\\
\node[2]{M}\arrow{n,r}{\gamma}\arrow{e,J}\node{N}
\end{diagram}
\end{equation*}
where $\pi:O(N)\to N$ denotes the projection in the bundle $O(N)|M$ and $\pi_{Gr}:Gr_n(N)\to N$ the projection in the Grassmann bundle. Notice that $\rho|O(N,M)$ is a constant map on the fibers, $\rho(u)=T_{\pi(u)}M$.

Considering the Riemannian metric $\tilde{g}$ on $M$, which may, in general, differ from $g$, we study the harmonicity of a map
\begin{equation}\label{eq:Gaussmap}
\gamma:(M,\tilde{g})\to (Gr_n(N),g_{Gr}).
\end{equation}  

Harmonicity of a Gauss map \eqref{eq:Gaussmap}, up to the knowledge of the author, was studied only in the case when $\tilde{g}=g$ \cite{rv,jr}. The harmonicity of the Gauss map for a distribution was studied in the full generality in \cite{ggv}. 

For a $(1,1)$--tensor $T$ on $N$ let
\begin{equation*}
\hat{T}_{\xi}=\mu_{\ast u}\bar{T}_u,\quad \mu(u)=\xi.
\end{equation*}
Moreover, we put
\begin{equation*}
X^{h,Gr}_{\xi}=\mu_{\ast u}(X^h),\quad \mu(u)=\xi.
\end{equation*}
Then, $\hat{T}$ and $X^{h,Gr}$ are well--defined vector fields on $Gr_n(N)$. Notice that $\hat{T}=\hat{T_{\mathfrak{m}}}$, i.e. $\hat{T_{\mathfrak{h}}}=0$. Thus, it is clear that the Levi--Civita connection of $Gr_n(N)$ is equal
\begin{align*}
\nabla^{Gr}_{X^{h,Gr}}Y^{h,Gr} &=(\nabla_X Y)^{h,Gr}-\frac{1}{2}\widehat{R(X,Y)_{\mathfrak{m}}},\\
\nabla^{Gr}_{X^{h,Gr}}\hat{T_{\mathfrak{m}}} &=\widehat{\nabla '_X T_{\mathfrak{m}}}+\frac{1}{2}R_{T_{\mathfrak{m}}}(X)^{h,Gr},\\
\nabla^{Gr}_{\hat{T_{\mathfrak{m}}}}\hat{T'_{\mathfrak{m}}} &=0.
\end{align*}

The Levi--Civita connection $\nabla^{Gr}$ induces the unique connection $\nabla^{\gamma}$ in the pull--back bundle $\gamma^{-1}(T\,Gr_n(N))\to M$ which satisfies the condition
\begin{equation*}
\nabla^{\gamma}_X (E\circ\varphi)=\nabla^{Gr}_{\gamma_{\ast} X}E,\quad X\in TM, \quad E\in\Gamma(T\,Gr_n(N)).
\end{equation*} 
Then, $\gamma$ is harmonic if its tension field $\tau(\gamma)$ vanishes, where
\begin{equation*}
\tau(\gamma)=\sum_A \nabla^{\gamma}_{\tilde{e_A}}\gamma_{\ast}\tilde{e_A}-\gamma_{\ast}(\tilde{\nabla}_{\tilde{e_A}}\tilde{e_A})
\end{equation*}
and $(\tilde{e_a})$ is an orthonormal basis of $M$ with respect to $\tilde{g}$ and $\tilde{\nabla}$ is the Levi--Civita connection of $\tilde{g}$. Notice that
\begin{equation*}
\gamma_{\ast} X=X^{h,Gr}\circ\gamma+\widehat{S_X}\circ\gamma,
\end{equation*}
which follows form the fact that
\begin{equation*}
\sigma_{\ast}X=X^h\circ\sigma+\sum_i (\nabla_X \sigma_i)^{v,i}
\end{equation*}
for a section of $O(N)$ and that $\gamma=\mu\circ\sigma^{ad}$, where $\sigma^{ad}$ is the section of $O(M,N)$. Therefore, we get the formula for the tension field,
\begin{equation*}
\tau(\gamma)=\sum_A (\nabla_{\tilde{e_A}}\tilde{e_A}-\tilde{\nabla}_{\tilde{e_A}}\tilde{e_A}
+R_{S_{\tilde{e_A}}}(\tilde{e_A}))^{h,Gr}+\widehat{\nabla '_{\tilde{e_A}}S_{\tilde{e_A}}}-\widehat{S_{\tilde{\nabla}_{\tilde{e_A}}\tilde{e_A}}}.
\end{equation*}
In particular, $\gamma$ is harmonic if and only if
\begin{align}
&\sum_A\Pi(\tilde{e_A},\tilde{e_A})+R_{S_{\tilde{e_A}}}(\tilde{e_A})^{\bot}=0, \tag{H1}\label{eq:H1}\\
&\sum_A \nabla '_{\tilde{e_A}}\tilde{e_A}-\tilde{\nabla}_{\tilde{e_A}}\tilde{e_A}
+R_{S_{\tilde{e_A}}}(\tilde{e_A})^{\top}=0,\tag{H2}\label{eq:H2}\\
&\nabla '_{\tilde{e_A}}S_{\tilde{e_A}}-S_{\tilde{\nabla}_{\tilde{e_A}}\tilde{e_A}}=0\tag{H3}\label{eq:H3}.
\end{align}

Now, we can state and prove the main theorem.
\begin{thm}
$O(M,N)$ is minimal in $O(N)$ if and only if the Gauss map $\gamma:(M,\tilde{g})\to (Gr_n(N),g_{Gr})$ of $M$ is harmonic.
\end{thm}
\begin{proof}
First, we compute the scalar product of the second fundamental form on horizontal vectors with orthogonal vectors $Z^{h}$, $Z\in T^{\bot}M$, and $\tilde{T}=\bar{T}+S_T^h$, for $T\in\mathfrak{m}(TN|M)$. Clearly
\begin{equation}\label{eq:PiwithZh}
g_{SM}(\Pi(X^{h'},Y^{h'}),Z^h)=g(\Pi(X,Y)+\frac{1}{2}(R_{S_X}(Y)+R_{S_Y}(X))^{\bot},Z).
\end{equation}
Put
\begin{equation*}
V=\nabla '_XY+\nabla '_YX+(R_{S_X}(Y)+R_{S_Y}(X))^{\top},\quad L=\nabla '_XS_Y+\nabla '_YS_X.
\end{equation*}
Then, by Lemma \ref{lem:propoperators},
\begin{equation}\label{eq:PiwithTbar}
\begin{split}
g_{SM}(\Pi(X^{h'},Y^{h'}),\tilde{T}) &=\frac{1}{2}(\skal{S_{P^{-1}(V)}}{T}+\skal{L}{T}+\skal{S_{P^{-1}(S_L)}}{T}\\
&-g(S_{S_{P^{-1}(V)}},S_T)+g(S_L,S_T)+g(S_{S_{P^{-1}(S_L)}},S_T))\\
&=\frac{1}{2}(g(V,S_T)+\skal{L}{T}-\skal{S_{S_L}}{T}-g(S_L,S_T))\\
&=\frac{1}{2}\skal{L-S_V}{T}.
\end{split}
\end{equation}
Let $(\tilde{e_A})$ be an orthonormal basis on $M$ with respect to $\tilde{g}$. Then, $(\tilde{e_A}^{h'})$ is an orthonormal basis of the horizontal distribution $\mathcal{H}'$, since
\begin{equation*}
g_{SM}(\tilde{e_A}^{h'},\tilde{e_B}^{h'})=\tilde{g}(\tilde{e_A},\tilde{e_B})=\delta_{AB}.
\end{equation*}
By \eqref{eq:PiwithZh} and \eqref{eq:PiwithTbar} the harmonicity of $\gamma$ is equivalent to the following conditions
\begin{align}
&\sum_A \Pi(\tilde{e_A},\tilde{e_A})+R_{S_{\tilde{e_A}}}(\tilde{e_A})^{\bot}=0,\tag{M1}\label{eq:M1}\\
&\sum_A \nabla '_{\tilde{e_A}}S_{\tilde{e_A}}
-S_{\nabla '_{\tilde{e_A}}{\tilde{e_A}}}-S_{R_{S_{\tilde{e_A}}}(\tilde{e_A})^{\top}}=0.
\tag{M2}\label{eq:M2}
\end{align}
Condition \eqref{eq:M1} is just the condition \eqref{eq:H1}. Moreover, \eqref{eq:H2} and \eqref{eq:H3} imply \eqref{eq:M2}. To show the converse, assume \eqref{eq:M1} and \eqref{eq:M2} hold. It suffices to show that \eqref{eq:H2} is satisfied. By Codazzi equation (see Lemma \ref{lem:tensordecomp}) and Lemma \ref{lem:propoperators}
\begin{align*}
\sum_A g(R_{S_{\tilde{e_A}}}(\tilde{e_A}),Z) &=\sum_A \skal{R(\tilde{e_A},Z)}{S_{\tilde{e_A}}}\\
&=\sum_A \skal{\nabla '_{\tilde{e_A}}S_Z-S_{\nabla '_{\tilde{e_A}}Z}}{S_{\tilde{e_A}}}
-\skal{\nabla '_ZS_{\tilde{e_A}}-S_{\nabla '_Z\tilde{e_A}}}{S_{\tilde{e_A}}}.
\end{align*}
Hence, by Lemma \ref{lem:tildeg}, equation \eqref{eq:GilMedranoeq} and \eqref{eq:M2} we have
\begin{align*}
\sum_A g(\tilde{\nabla}_{\tilde{e_A}}\tilde{e_A}-\nabla '_{\tilde{e_A}}\tilde{e_A},P(Z)) &=
\sum_A g((\nabla '_{\tilde{e_A}}P)\tilde{e_A},Z)-\frac{1}{2}g((\nabla '_ZP)\tilde{e_A},\tilde{e_A})\\
&=\sum_A\big(\skal{\nabla '_{\tilde{e_A}}S_{\tilde{e_A}}-S_{\nabla '_{\tilde{e_A}}\tilde{e_A}}}{S_Z}
+\skal{\nabla '_{\tilde{e_A}}S_Z-S_{\nabla '_{\tilde{e_A}}Z}}{S_{\tilde{e_A}}}\\
&-\skal{\nabla '_ZS_{\tilde{e_A}}-S_{\nabla '_Z\tilde{e_A}}}{S_{\tilde{e_A}}}\big)\\
&=\sum_A \skal{\nabla '_{\tilde{e_A}}S_{\tilde{e_A}}-S_{\nabla '_{\tilde{e_A}}\tilde{e_A}}}{S_Z}+g(R_{S_{\tilde{e_A}}}(\tilde{e_A}),Z)\\
&=g(S_{R_{S_{\tilde{e_A}}}(\tilde{e_A})^{\top}},S_Z)+g(R_{S_{\tilde{e_A}}}(\tilde{e_A}),Z)\\
&=\sum_A g(R_{S_{\tilde{e_A}}}(\tilde{e_A})^{\top},P(Z)).
\end{align*}
Therefore \eqref{eq:H2} holds.
\end{proof}

\begin{rem}
It seems that the correlation of minimality of $O(M,N)$ and the harmonicity of associated section of the Grassmann bundle is a more general fact, which holds for any restriction of the structure group of the principal bundle and the homogeneous bundle associated with this bundle. This correlation is studied by the author independently \cite{nie3}.
\end{rem}


\begin{thebibliography}{10}
\bibitem{cl} L. A. Cordero, M. de Leon, On the curvature of the induced Riemannian metric on the frame bundle of a Riemannian manifold, J. Math. Pures Appl. (9) 65 (1986), no. 1, 81--91.
\bibitem{gm} O. Gil--Medrano, Relationship between volume and energy of vector fields, Differential Geom. Appl. 15 (2001),  no. 2, 137–152.
\bibitem{ggv} O. Gil--Medrano, J. C. Gonzalez--Davila, L. Vanhecke, Harmonicity and minimality of oriented distributions, Israel J. Math. 143 (2004), 253--279.
\bibitem{jr} G. Jensen, M. Rigoli, Harmonic Gauss maps. Pacific J. Math.  136  (1989),  no. 2, 261--282,
\bibitem{kn} S. Kobayashi, K. Nomizu, Foundations of differential geometry, vol. II, Interscience Publishers, Wiley and Sons, 1969. 
\bibitem{ks0} O. Kowalski, M. Sekizawa, On curvatures of linear frame bundles with naturally lifted metrics, Rend. Semin. Mat. Univ. Politec. Torino 63 (2005), no. 3, 283--295.
\bibitem{ks1} O. Kowalski, M. Sekizawa, On the geometry of orthonormal frame bundles. Math. Nachr. 281 (2008), no. 12, 1799--1809.
\bibitem{ks2} O. Kowalski, M. Sekizawa, On the geometry of orthonormal frame bundles. II, Ann. Global Anal. Geom. 33 (2008), no. 4, 357--371.
\bibitem{mok} K. P. Mok, On the differential geometry of frame bundles of Riemannian manifolds, J. Reine Angew. Math. 302 (1978), 16--31.
\bibitem{nie2} K. Niedzia\l omski, On the frame bundle adapted to a distribution, in preparation.
\bibitem{nie3} K. Niedzia\l omski, Geometry of G--structures, in preparation.
\bibitem{rv} E. Ruh, J. Vilms, The tension field of the Gauss map, Trans. Amer. Math. Soc. 149 (1970), 569--573.
\end{thebibliography}
\end{document}